\documentclass[12pt, reqno]{amsart}
\usepackage{amsmath,amssymb,amsthm}

\usepackage{amsmath, amssymb}
\usepackage{amsthm, amsfonts, mathrsfs}
\usepackage{mathptmx}
\usepackage{fullpage}
\usepackage{amsfonts,graphicx}
\numberwithin{equation}{section}
\usepackage[colorlinks=true, pdfstartview=FitV, linkcolor=blue, citecolor=blue, urlcolor=blue]{hyperref}

\newcommand{\R}{{\mathbb{R}}}
\def\div{ \hbox{\rm div}\,  }
\def\u{ \mathbf{u} }
\def\n{ \mathbf{n} }
\def\b{ \mathbf{B} }
\def\T{ \mathbb{T} }
\newcommand{\Z}{{\mathbb{Z}}}

\def\nn{\nonumber}

\newtheorem{theorem}{Theorem}[section]

\newtheorem{lemma}[equation]{Lemma}

\theoremstyle{remark}
\newtheorem{remark}[theorem]{Remark}
\numberwithin{equation}{section}
\newcommand{\norm}[2]{\left\lVert #1 \right\rVert_{#2}}

\begin{document}

\title[Stability for the 2D  MHD system ]{Stability for the 2D incompressible MHD
equations with only magnetic diffusion}
\author[Xiaoping Zhai]{Xiaoping Zhai }

\address{ School  of Mathematics and Statistics, Guangdong University of Technology,
	Guangzhou, 510520,   China}

\email{pingxiaozhai@163.com }
\date{\today}

\begin{abstract}
This paper presents a global stability result on perturbations near a
background magnetic field to the 2D incompressible magnetohydrodynamic (MHD) equations with only magnetic diffusion on the periodic domain. The stability result provides a
significant example for the stabilizing effects of the magnetic
field on electrically conducting fluids. In addition, we obtain
an explicit large-time decay rate of the solutions.
	\end{abstract}
\maketitle

\section{ Introduction and main result}
In this paper, we are concerned with the stability of the smooth solutions to
 the following inviscid MHD system
\begin{eqnarray}\label{m1}
\left\{\begin{aligned}
&\partial_t \u+ \u\cdot\nabla \u+\nabla p=\b\cdot\nabla \b,\\
&\partial_t \b-\Delta\b+\u\cdot\nabla \b =\b\cdot\nabla \u,\\
&\div \u =\div \b =0,\\
&(\u,\b)|_{t=0}=(\u_0,\b_0).
\end{aligned}\right.
\end{eqnarray}
Here, $x=(x_1,x_2)\in \T^2$ and $t\geq 0$ are the space and time variables, respectively. The unknown $\u$ is the velocity field, $\b$
is the magnetic field, $ {p}$ is the scalar pressure, respectively.
The MHD system with only magnetic diffusion models many significant phenomenon such as the magnetic reconnection in astrophysics and geomagnetic dynamo in geophysics (see, e.g., \cite{EPTF}). For more physical background, we refer to \cite{HC}, \cite{TGCDP}, \cite{DL}, \cite{LDEM}, \cite{ST}.

Due to the nonlinear interaction between
the fluid velocity and the magnetic field, the MHD equations can accommodate much richer phenomena
than the Navier-Stokes or Euler equations alone. One significant example is that the magnetic field can actually stabilize the fluid motion \cite{bee}. The MHD system has always been of great interest in mathematics. There are many works  devoted to the global well-posedness
of the MHD system with whole diffusion or partial diffusion, see \cite{HAPZ}, \cite{CBCS}, \cite{CL}--\cite{CWY}, \cite{dongboqing}, \cite{fefferman1}--\cite{JNWXY},
\cite{lijinlu} --\cite{PZZ}, \cite{REN}, \cite{RXZ2}, \cite{wujiahong1}, \cite{wujiahong2}, \cite{wei1}, \cite{wei2}, \cite{LXPZ}, \cite{TZ} and the references therein.

In the case of our consideration, namely the incompressible MHD system with zero
viscosity and  positive resistivity, it is still an open problem whether or not there exists global
classical solutions even in $\R^2$ for generic smooth small initial data. It  made some progress only recently by Zhou and Zhu \cite{zhuyijmp}, Wei and Zhang \cite{wei3}  on a periodic domain, see also an improvement work of \cite{wei3} in \cite{yeweikui}.
More precisely, Zhou and Zhu \cite{zhuyijmp} studied the global existence of classical solutions to system \eqref{m1} in $\T^2$. The proof \cite{zhuyijmp} depends heavily on a time-weighted energy estimate and the assumptions that the initial magnetic field is close enough to an equilibrium state and the initial data have reflection symmetry.
As the first global well-posedness result, Wei and Zhang \cite{wei3} proved the global solutions of \eqref{m1} in $\T^2$ with small initial data,
without the non-trivial background magnetic field assumption. However, the solutions constructed in \cite{wei3} may grow in time, especially,
the  $\|\nabla\u\|_{L^\infty}$ will grow exponentially in time. Hence, the stability of the  solutions constructed in \cite{wei3} is still unknown.

Inspired by \cite{gerard}, \cite{zhangzhifei}, \cite{wei3},  and \cite{zhuyijmp}, the contribution of this paper is the global existence and uniqueness of solutions
of \eqref{m1} with sufficiently smooth initial data $(\u_0, \b_0) $ close to the equilibrium state
$(\mathbf{0}, \n)$, where ${\n}\in\R^2$ satisfies the so called Diophantine condition: for any $\mathbf{k}\in\Z^2\setminus \{0\},$
\begin{align}\label{diufantu}
|\n\cdot \mathbf{k}|\ge \frac{c}{|\mathbf{k}|^r}, \quad\hbox{for some $c>0$ and $r>1$.}
\end{align}
Moreover, as demonstrated in \cite{gerard} and \cite{zhangzhifei},
for almost all vectors in ${\mathbb R}^2$,  there exist $c=c(\n)$ so  that the Diophantine condition
(\ref{diufantu}) holds.

For the simplicity, we still use the notation $\b$ to denote the perturbation $\b -{\n}$. Hence, the perturbed equations can be rewritten  into
\begin{eqnarray}\label{m}
\left\{\begin{aligned}
&\partial_t \u+ \u\cdot\nabla \u+\nabla p={\n}\cdot\nabla \b+\b\cdot\nabla \b,\\
&\partial_t \b-\Delta\b+\u\cdot\nabla \b =\mathbf{n}\cdot\nabla \u+\b\cdot\nabla \u,\\
&\div \u =\div \b =0,\\
&(\u,\b)|_{t=0}=(\u_0,\b_0).
\end{aligned}\right.
\end{eqnarray}

The   main result of the paper is stated as follows.
\begin{theorem}\label{dingli}
Assume $\n$ satisfies the Diophantine condition (\ref{diufantu}).
Let  $\alpha>0, \beta>0$ be  two arbitrarily fixed constants.
  For any  ${N}\ge (2\beta+3)r+\alpha+2\beta+5$ with $r>1$, and $(\u_0,\b_0)\in H^{N}(\T^2)$ with
\begin{align*}
\int_{\T^2}\u_0\,dx=\int_{\T^2}\b_0\,dx=0.
\end{align*}
 If there exists a small constant $\varepsilon$ such that
\begin{align*}
\norm{\u_0}{H^{N}}+\norm{\b_0}{H^{N}}\le\varepsilon.
\end{align*}
Then the system \eqref{m} admits a  global  solution $( \u,\b)\in C([0,\infty );H^{N})$. Moreover, for any $t\ge 0$ and $r+\alpha+3\le\gamma\le N $,
there holds
\begin{align*}
\norm{\u(t)}{H^{\gamma}}+\norm{\b(t)}{H^{\gamma}}\le C(1+t)^{-\frac{({N}-\gamma)(\beta+1)}{{N}-r-\alpha-3}}.
\end{align*}
\end{theorem}

\begin{remark}
Compared to \cite{zhuyijmp}, we have removed the reflection symmetry assumption on the initial data.
Moreover, we obtain  an explicit large-time decay rate of the solutions.
\end{remark}

\begin{remark}
It should be mentioned that our theorem still be valid for the 2D incompressible viscous non-resistive MHD system. Interested readers can find similar result in \cite{zhangzhifei} for 3D case.
\end{remark}
\begin{remark}
Our methods can be used to other related models. Similar result for the
compressible system will be presented in a forthcoming paper.
\end{remark}

\begin{remark}
Whether the Diophantine condition \eqref{diufantu} can be removed is a  challenged open problem. This is left in the future work.
\end{remark}

\begin{remark}
It seems difficult to generalize our result to the case of the whole space $\R^2$,
since the Diophantine condition \eqref{diufantu} plays a crucial role in our proof.
\end{remark}

\section{The proof of the theorem}
The  proof of the Theorem \ref{dingli} relies heavily on the following two lemmas.
\begin{lemma}\label{diu}
Let ${\n}\in\R^2$ satisfy the Diophantine condition \eqref{diufantu}.

\begin{itemize}
  \item For any $s\in R$, there holds
  \begin{equation}\label{poin1}
\|f\|_{H^{s}}\le C\|{\n}\cdot\nabla f\|_{H^{s+r}},\ \ if\ \int_{\T^2}f\,dx=0.
\end{equation}
  \item For any $s>0$, one can remove the zero-mean condition by using homogeneous norms. Precisely, if $s>0$, there holds, for any $f$, that
\begin{equation}\label{poin2}
\|f\|_{\dot{H}^{s}}\le C\|{\n}\cdot\nabla f\|_{H^{s+r}}.
\end{equation}
\end{itemize}
\end{lemma}

\begin{proof}
We give the proof for completeness. By Plancherel formula, we have
\begin{align*}
\|{\n}\cdot\nabla f\|_{H^{s+r}}^2
=&\sum_{\mathbf{k}\in\Z^2}(1+|\mathbf{k}|^2)^{s+r}|\n\cdot\mathbf{k}|^2|\hat{f}|^2\nn\\
=&\sum_{\mathbf{k}\in\Z^2\setminus \{0\}}(1+|\mathbf{k}|^2)^{s+r}|\n\cdot\mathbf{k}|^2|\hat{f}|^2\nn\\
\ge&c\sum_{\mathbf{k}\in\Z^2\setminus \{0\}}(1+|\mathbf{k}|^2)^{s+r}|\mathbf{k}|^{-2r}|\hat{f}|^2\nn\\
\ge&c\sum_{\mathbf{k}\in\Z^2\setminus \{0\}}(1+|\mathbf{k}|^2)^{s}|\hat{f}|^2.
\end{align*}
So if $\int_{\T^2} f=0$, we have \eqref{poin1} since $\hat{f}(0)=0$. If $s>0$, we have \eqref{poin2}.
\end{proof}

\begin{lemma}\label{daishu}{\rm(\cite{kato})}
Let $s\ge 0$. Then there exists a constant $C$ such that, for any $f,g\in {H^{s}}(\T^2)\cap {L^\infty}(\T^2)$, we have
\begin{equation*}
\|fg\|_{H^{s}}\le C(\|f\|_{L^\infty}\|g\|_{H^{s}}+\|g\|_{L^\infty}\|f\|_{H^{s}}).
\end{equation*}
\end{lemma}
Now, we begin to prove the main theorem.
\subsection{$L^2$ energy estimate}
Firstly, denote $\langle a,b\rangle$ the $L^2(\T^2)$ inner product of $a$ and $b$.
A standard energy estimate gives
\begin{align}\label{ping1}
&\frac12\frac{d}{dt}(\norm{\u}{L^2}^2+\norm{\b}{L^2}^2)+\norm{\nabla\b}{L^2}^2=0
\end{align}
where  we used the  following cancellations
\begin{align*}
&\big\langle \u\cdot\nabla \u,\u\big\rangle
   =\big\langle \u\cdot\nabla \b,\b\big\rangle=0,
  \quad \big\langle\nabla p,\u\big\rangle=0
,\nn\\
&\big\langle \b\cdot\nabla \b,\u\big\rangle
       +\big\langle \b\cdot\nabla \u,\b\big\rangle=0,\quad\big\langle{\n}\cdot\nabla \b, \u\big\rangle
   +\big\langle{\n}\cdot\nabla \u, \b \big\rangle=0.
\end{align*}

\subsection{High order energy estimate}
  Denote $D \stackrel{\mathrm{def}}{=}\sqrt{-\Delta}$.
  We operate
$D^\alpha$
on the first two equations respectively
and take the scalar product of them
with $D^\alpha \u$ and $D^\alpha \b$ respectively,
 add them together and then sum the result over $|\alpha|\leq m$.
  We obtain
\begin{align}\label{3.8}
&\frac{1}{2}\frac{d}{dt}\big(\left\|\u\right\|_{H^m}^2+\left\|\b\right\|_{H^m}^2\big)
        +\left\|\nabla\b\right\|_{H^m}^2\nn\\
&\quad=-\sum_{0<|\alpha|\leq m}\left\langle[D^{\alpha},\u\cdot\nabla ]\u,D^{\alpha}\u\right\rangle+\sum_{0<|\alpha|\leq m}\left\langle[D^{\alpha},\b\cdot\nabla ]\b,D^{\alpha}\u\right\rangle\nn\\
&\qquad-\sum_{0<|\alpha|\leq m}\left\langle[D^{\alpha},\u\cdot\nabla ]\b,D^{\alpha}\b\right\rangle+\sum_{0<|\alpha|\leq m}\left\langle[D^{\alpha},\b\cdot\nabla ]\u,D^{\alpha}\b\right\rangle
\end{align}
where we used the estimate \eqref{ping1} and the following cancellations
\begin{align*}
&\left\langle \u\cdot\nabla D^{\alpha} \u,D^{\alpha}\u\right\rangle=\left\langle  D^{\alpha} \nabla p,D^{\alpha}\u\right\rangle=\left\langle \u\cdot\nabla D^{\alpha} \b,D^{\alpha}\b\right\rangle=0,
\nn\\
&\left\langle \b\cdot\nabla D^{\alpha} \b,D^{\alpha}\u\right\rangle+\left\langle \b\cdot\nabla D^{\alpha} \u,D^{\alpha}\b\right\rangle=0,\nn\\
&\left\langle D^{\alpha}({\n}\cdot\nabla \b),D^{\alpha}\u\right\rangle+\left\langle D^{\alpha}({\n}\cdot\nabla \u),D^{\alpha}\b\right\rangle=0.
\end{align*}
In view of the well-known calculus inequality,
\begin{align*}
 \sum_{|\alpha|\leq m}\left\|D^{\alpha}(fg)-(D^{\alpha}f)g\right\|_{L^2}
 \leq C\big(\|f\|_{H^{m-1}}\|\nabla g\|_{L^{\infty}}
        +\|f\|_{L^{\infty}}\|g\|_{H^{m}}\big),
\end{align*}
we have
\begin{align*}
\norm{[D^{\alpha},\u\cdot\nabla ]\u}{L^2}+\norm{[D^{\alpha},\b\cdot\nabla ]\b}{L^2}\le C (\norm{\nabla \u}{L^\infty}\norm{D^{\alpha} \u}{L^2}+ \norm{\nabla \b}{L^\infty}\norm{D^{\alpha} \b}{L^2}),
\end{align*}
\begin{align*}
\norm{[D^{\alpha},\u\cdot\nabla ]\b}{L^2}+\norm{[D^{\alpha},\b\cdot\nabla ]\u}{L^2}\le C(\norm{\nabla \u}{L^\infty}\norm{D^{\alpha}\b}{L^2}+\norm{D^{\alpha} \u}{L^2}\norm{\nabla \b}{L^\infty}).
\end{align*}
Inserting the above estimates into \eqref{3.8}, we obtain
\begin{align}\label{3.16}
&\frac12\frac{d}{dt}\big(\|\u\|_{H^{m}}^2+\|\b\|_{H^{m}}^2\big)
 +\|\nabla\b\|_{H^{m}}^2
       \nonumber\\
 &\quad\leq C\big(\|\nabla\u\|_{L^{\infty}}+\|\nabla\b\|_{L^{\infty}}\big)
  \big(\|\u\|_{H^{m}}^2+\|\b\|_{H^{m}}^2\big).
\end{align}

\subsection{A key Lemma}
The following lemma which relies heavily on the structural characteristics of the system \eqref{m} is crucial to get the time decay of the velocity field.
\begin{lemma}\label{ping13}
For any ${N}\ge r+\alpha+4$ with $r>1$ and $\alpha>0$. Assume that
\begin{align}\label{ping14}
\sup_{t\in[0,T]}(\norm{\u}{H^{N}}+\norm{\b}{H^{N}})\le \delta,
\end{align}
for some $0<\delta<1.$ Then there holds that
\begin{align}\label{ping15}
&\norm{{\n}\cdot\nabla \u}{H^{r+\alpha+2}}^2-\sum_{0\le s\le r+\alpha+2}\frac{d}{dt}\big\langle{D^s}\b,{D^s}({\n}\cdot\nabla \u)\big\rangle\le C\norm{ \nabla\b}{H^{r+\alpha+3}}^2.
\end{align}

\end{lemma}

\begin{proof}
Applying ${D^s} (0\le s\le r+\alpha+2) $ to the second equation of \eqref{m}, and multiplying it by ${D^s}({\n}\cdot\nabla \u)$
then integrating over $\T^2$, we obtain
\begin{align}\label{ping16}
\norm{{D^s}({\n}\cdot\nabla \u)}{L^2}^2
=&{\big\langle{D^s}\partial_t \b,{D^s}({\n}\cdot\nabla \u)\big\rangle}-{\big\langle{D^s} \Delta\b,{D^s}({\n}\cdot\nabla \u)\big\rangle}\nn\\
&+{\big\langle{D^s} (\u\cdot\nabla \b),{D^s}({\n}\cdot\nabla \u)\big\rangle}-{\big\langle{D^s}(\b\cdot\nabla \u),{D^s}({\n}\cdot\nabla \u)\big\rangle}\nn\\
 \stackrel{\mathrm{def}}{=}&I_1+I_2+I_3+I_4.
\end{align}
Thanks to the H\"older inequality, Young's inequality, and the embedding relation,
 we have
\begin{align}\label{ping16+1}
I_2\le& C\norm{{D^s}\Delta\b}{L^2}\norm{{D^s}({\n}\cdot\nabla \u)}{L^2}\nn\\
\le&\frac{1}{16}\norm{{D^s}({\n}\cdot\nabla \u)}{L^2}^2+C\norm{\Delta \b}{H^s}^2\nn\\
\le&\frac{1}{16}\norm{{D^s}({\n}\cdot\nabla \u)}{L^2}^2+C\norm{\nabla\b}{H^{s+1}}^2\nn\\
\le&\frac{1}{16}\norm{{D^s}({\n}\cdot\nabla \u)}{L^2}^2+C\norm{\nabla\b}{H^{r+\alpha+3}}^2.
\end{align}
Similarly,  using Lemmas \ref{diu} and \ref{daishu}, we obtain
\begin{align}\label{ping17}
I_3\le& C\norm{{D^s}(\u\cdot\nabla \b)}{L^2}\norm{{D^s}({\n}\cdot\nabla \u)}{L^2}\nn\\
\le& C(\norm{\u}{L^\infty}\norm{\nabla \b}{H^s}+\norm{\nabla \b}{L^\infty}\norm{ \u}{H^s} )\norm{{D^s}({\n}\cdot\nabla \u)}{L^2}\nn\\
\le&\frac{1}{16}\norm{{D^s}({\n}\cdot\nabla \u)}{L^2}^2+C(\norm{\u}{H^{2+\alpha}}^2\norm{\nabla \b}{H^s}^2+\norm{\nabla \b}{H^{2+\alpha}}^2\norm{ \u}{H^s}^2 )\nn\\
\le&\frac{1}{16}\norm{{D^s}({\n}\cdot\nabla \u)}{L^2}^2+C\norm{\n\cdot\nabla\u}{H^{r+\alpha+2}}^2\norm{ \nabla\b}{H^{s}}^2+C\norm{ \nabla\b}{H^{2+\alpha}}^2\norm{ \u}{H^N}^2 \nn\\
\le&\frac{1}{16}\norm{{D^s}({\n}\cdot\nabla \u)}{L^2}^2+C\delta^2\norm{ \n\cdot\nabla\u}{H^{r+\alpha+2}}^2+C\delta^2\norm{ \nabla\b}{H^{r+\alpha+3}}^2,
\end{align}
and
\begin{align}\label{ping19-1}
I_4\le& C\norm{{D^s}(\b\cdot\nabla \u)}{L^2}\norm{{D^s}({\n}\cdot\nabla \u)}{L^2}\nn\\
\le& C(\norm{\b}{L^\infty}\norm{\nabla \u}{H^s}+\norm{\nabla \u}{L^\infty}\norm{ \b}{H^s} )\norm{{D^s}({\n}\cdot\nabla \u)}{L^2}\nn\\
\le& C(\norm{\b}{H^2}\norm{ \u}{H^{s+1}}+\norm{\nabla \u}{H^2}\norm{ \b}{H^s} )\norm{{D^s}({\n}\cdot\nabla \u)}{L^2}\nn\\
\le&\frac{1}{16}\norm{{D^s}({\n}\cdot\nabla \u)}{L^2}^2+C\norm{\u}{H^N}^2(\norm{ \b}{H^{s}}^2+\norm{ \b}{H^{2}}^2)\nn\\
\le&\frac{1}{16}\norm{{D^s}({\n}\cdot\nabla \u)}{L^2}^2+C\delta^2\norm{ \b}{H^{r+\alpha+2}}^2.
\end{align}
Due to $\int_{\T^2}\b_0\,dx=0,$  there holds
\begin{align}\label{pindg}
\norm{ \b}{H^{r+\alpha+2}}^2\le C\norm{ \nabla\b}{H^{r+\alpha+3}}^2
\end{align}
from which we can further get that
\begin{align}\label{ping19}
I_4
\le&\frac{1}{16}\norm{{D^s}({\n}\cdot\nabla \u)}{L^2}^2+C\delta^2\norm{ \nabla\b}{H^{r+\alpha+3}}^2.
\end{align}

Finally, we have to bound $I_1$. In fact, exploiting the first equation in \eqref{m}, we can rewrite this term into
\begin{align}\label{ping21}
I_1=&\frac{d}{dt}\big\langle{D^s}\b,{D^s}({\n}\cdot\nabla \u)\big\rangle-\big\langle{D^s}\b,{D^s}({\n}\cdot\nabla \partial_t\u)\big\rangle\nn\\
=&\frac{d}{dt}\big\langle{D^s}\b,{D^s}({\n}\cdot\nabla \u)\big\rangle+\big\langle{D^s}({\n}\cdot\nabla\b),{D^s} \partial_t\u\big\rangle\nn\\
=&\frac{d}{dt}\big\langle{D^s}\b,{D^s}({\n}\cdot\nabla \u)\big\rangle+\big\langle{D^s}({\n}\cdot\nabla \b),{D^s}({\n}\cdot\nabla \b)\big\rangle\\
&+\big\langle{D^s}({\n}\cdot\nabla \b),{D^s}(\b\cdot\nabla \b)\big\rangle-\big\langle{D^s}({\n}\cdot\nabla \b),{D^s}(\u\cdot\nabla \u)\big\rangle\nn
\end{align}
where we  used the cancellation
\begin{align*}
\big\langle{D^s}({\n}\cdot\nabla\b),{D^s} \nabla p\big\rangle=0.
\end{align*}
It follows from the H\"older inequality directly that
\begin{align}\label{pingping2-11}
\big\langle{D^s}({\n}\cdot\nabla \b),{D^s}({\n}\cdot\nabla \b)\big\rangle
\le& C\norm{{D^s}({\n}\cdot\nabla \b)}{L^2}\norm{{D^s}({\n}\cdot\nabla \b)}{L^2}\nn\\
\le& C\norm{\nabla\b}{H^{s}}^2\nn\\
\le& C\norm{\nabla\b}{H^{r+\alpha+3}}^2.
\end{align}
Using Lemma \ref{daishu}, the third term on the right hand side of \eqref{ping21} can be bounded as
\begin{align}\label{pingping2}
\big\langle{D^s}({\n}\cdot\nabla \b),{D^s}(\b\cdot\nabla \b)\big\rangle
\le& C\norm{{D^s}({\n}\cdot\nabla \b)}{L^2}(\norm{\b}{L^\infty}\norm{\nabla \b}{H^s}+\norm{\nabla \b}{L^\infty}\norm{ \b}{H^s} )\nn\\
\le& C\norm{\b}{H^N}\norm{\nabla \b}{H^s}^2\nn\\
\le& C\delta\norm{\nabla\b}{H^{r+\alpha+3}}^2.
\end{align}
In the same manner, we can deal with the last term  on the right hand side of \eqref{ping21}
\begin{align}\label{pingping3}
\big\langle{D^s}({\n}\cdot\nabla \b),{D^s}(\u\cdot\nabla \u)\big\rangle
\le& C\norm{{D^s}({\n}\cdot\nabla \b)}{L^2}(\norm{\u}{L^\infty}\norm{\nabla \u}{H^s}+\norm{\nabla \u}{L^\infty}\norm{ \u}{H^s} )\nn\\
\le& C\norm{{D^s}({\n}\cdot\nabla \b)}{L^2}(\norm{\u}{H^{1+\alpha}}\norm{\nabla \u}{H^s}+\norm{\nabla \u}{H^{1+\alpha}}\norm{ \u}{H^s} )\nn\\
\le& C\norm{\nabla \b}{H^s}\norm{\u}{H^N}\norm{\u}{H^{{2+\alpha}}}\nn\\
\le& C\norm{\nabla \b}{H^s}^2+C\delta^2\norm{\u}{H^{2+\alpha}}^2\nn\\
\le& C\norm{\nabla\b}{H^{r+\alpha+3}}^2+C\delta^2\norm{\n\cdot\nabla\u}{H^{r+\alpha+2}}^2.
\end{align}
Inserting \eqref{pingping2-11}--\eqref{pingping3} into \eqref{ping21} gives
\begin{align}\label{ping22}
\big\langle{D^s}\partial_t \b,{D^s}({\n}\cdot\nabla \u)\big\rangle
\le&\frac{d}{dt}\big\langle{D^s}\b,{D^s}({\n}\cdot\nabla \u)\big\rangle+C\norm{\nabla\b}{H^{r+\alpha+3}}^2+C\delta^2\norm{\n\cdot\nabla \u}{H^{r+\alpha+2}}^2.
\end{align}

Hence, if $\delta $ is small enough, plugging  \eqref{ping16+1}--\eqref{ping19} and \eqref{ping22} into \eqref{ping16},
we can arrive at \eqref{ping15}. This proves the lemma.
\end{proof}

\subsection{Complete the proof of the main theorem}
For any $(\u_0, \b_0)\in H^{N}(\T^2)$, the local well-posedness of \eqref{m1} can be  proved  by using the standard energy method, see also \cite{yeweikui} for a similar result. Thus, we may assume that there exist $T > 0$ and a unique solution
$(\u,\b)\in C([0,T];H^{N})$ of the system \eqref{m}. Furthermore, we may assume that
\begin{align}\label{ping23}
\sup_{t\in[0,T]}(\norm{\u}{H^{N}}+\norm{\b}{H^{N}})\le \delta,
\end{align}
for some $0<\delta<1$
 to be determined later.

Taking $m=r+\alpha+3$ in \eqref{3.16} gives
 \begin{align}\label{gan2}
&\frac12\frac{d}{dt}\big(\|\u\|_{H^{r+\alpha+3}}^2+\|\b\|_{H^{r+\alpha+3}}^2\big)
 +\|\nabla\b\|_{H^{r+\alpha+3}}^2\nn\\
 &\quad\leq C(\|\nabla\u\|_{L^{\infty}}+\|\nabla\b\|_{L^{\infty}}) \big(\|\u\|_{H^{r+\alpha+3}}^2+\|\b\|_{H^{r+\alpha+3}}^2\big).
\end{align}

Hence, let
 $A\ge 1+2C$ be a constant determined later, we infer from Lemma \ref{ping13} and  \eqref{gan2} that
\begin{align}\label{ping24}
&\frac{d}{dt}\left\{A(\norm{\u}{H^{r+\alpha+3}}^2+\norm{\b}{H^{r+\alpha+3}}^2)-\sum_{0\le s\le r+\alpha+2}\big\langle{D^s}\b,{D^s}({\n}\cdot\nabla \u)\big\rangle\right\}\nn\\
&\quad+A\norm{\nabla\b}{H^{r+\alpha+3}}^2+\norm{{\n}\cdot\nabla \u}{H^{r+\alpha+2}}^2\nn\\
&\le CA(\|\nabla\u\|_{L^{\infty}}+\|\nabla\b\|_{L^{\infty}}) \big(\|\u\|_{H^{r+\alpha+3}}^2+\|\b\|_{H^{r+\alpha+3}}^2\big),
\end{align}
provided that $\delta$  is small enough.

Thanks to the embedding relation $H^{\alpha+1}(\T^2)\hookrightarrow L^\infty(\T^2)$ and Lemma \ref{diu}, there holds
\begin{align}\label{gan3}
A\norm{\nabla \u}{L^\infty}\norm{\u}{H^{r+\alpha+3}}^2
&\le CA\norm{ \u}{H^{\alpha+2}}\norm{\u}{H^{r+\alpha+3}}^2\nn\\
&\le CA\norm{ \n\cdot\nabla\u}{H^{r+\alpha+2}}\norm{\u}{H^{r+\alpha+3}}^2\nn\\
&\le  \frac {1}{2}\norm{ \n\cdot\nabla\u}{H^{r+\alpha+2}}^2+CA^2\norm{\u}{H^{r+\alpha+3}}^4.
\end{align}
Similarly, using \eqref{pindg}, we get
\begin{align}\label{ping27}
\norm{\nabla \u}{L^\infty}\norm{\b}{H^{r+\alpha+3}}^2
\le& C\norm{ \u}{H^{\alpha+2}}\norm{\b}{H^{r+\alpha+3}}^2
\le C\delta\norm{\nabla\b}{H^{r+\alpha+3}}^2,\nn\\
\norm{\nabla \b}{L^\infty}\norm{\b}{H^{r+\alpha+3}}^2
\le& C\norm{ \b}{H^3}\norm{\b}{H^{r+\alpha+3}}^2
\le C\delta\norm{\nabla\b}{H^{r+\alpha+3}}^2,
\end{align}
and
\begin{align}\label{ping28}
A\norm{\nabla \b}{L^\infty}\norm{\u}{H^{r+\alpha+3}}^2
&\le
CA\norm{ \nabla\b}{H^{r+\alpha+3}}\norm{\u}{H^{r+\alpha+3}}^2\nn\\
&\le  \frac {A}{2}\norm{ \nabla\b}{H^{r+\alpha+3}}^2+C\norm{\u}{H^{r+\alpha+3}}^4.
\end{align}
Taking \eqref{gan3}--\eqref{ping28} into \eqref{ping24} implies that
\begin{align}\label{ping28+1}
&\frac{d}{dt}\left\{A(\norm{\u}{H^{r+\alpha+3}}^2+\norm{\b}{H^{r+\alpha+3}}^2)-\sum_{0\le s\le r+\alpha+2}\big\langle{D^s}\b,{D^s}({\n}\cdot\nabla \u)\big\rangle\right\}\nn\\
&\qquad+A\norm{\nabla\b}{H^{r+\alpha+3}}^2+\norm{{\n}\cdot\nabla \u}{H^{r+\alpha+2}}^2\nn\\
&\quad\le C(A^2+1)\norm{\u}{H^{r+\alpha+3}}^4+ CA\delta\norm{\nabla\b}{H^{r+\alpha+3}}^2.
\end{align}

For any $N\ge 2r+\alpha+4$, by the classical interpolation inequality in the Sobolev space, there holds
\begin{align}\label{ping25}
\norm{ \u}{H^{r+\alpha+3}}\le\norm{ \u}{H^{\alpha+2}}^{\frac12}\norm{ \u}{H^{{N}}}^{\frac12}\le C\delta^{\frac12}\norm{{\n}\cdot\nabla \u}{H^{r+\alpha+2}}^{\frac12}.
\end{align}

As a result,  we get from \eqref{ping28+1} that
\begin{align}\label{ping30}
&\frac{d}{dt}\left\{A(\norm{\u}{H^{r+\alpha+3}}^2+\norm{\b}{H^{r+\alpha+3}}^2)-\sum_{0\le s\le r+\alpha+2}\big\langle{D^s}\b,{D^s}({\n}\cdot\nabla \u)\big\rangle\right\}\nn\\
&\qquad+A\norm{\nabla\b}{H^{r+\alpha+3}}^2+\norm{{\n}\cdot\nabla \u}{H^{r+\alpha+2}}^2\nn\\
&\quad\le C(A^2+1)\delta^2\norm{{\n}\cdot\nabla \u}{H^{r+\alpha+2}}^2+CA\delta\norm{\nabla\b}{H^{r+\alpha+3}}^2.
\end{align}
Define
\begin{align*}
{\mathcal{D}(t)}=&A\norm{\nabla\b}{H^{r+\alpha+3}}^2+\norm{{\n}\cdot\nabla \u}{H^{r+\alpha+2}}^2,\nn\\
{\mathcal{E}(t)}=&A(\norm{\u}{H^{r+\alpha+3}}^2+\norm{\b}{H^{r+\alpha+3}}^2)-\sum_{0\le s\le r+\alpha+2}\big\langle{D^s}\b,{D^s}({\n}\cdot\nabla \u)\big\rangle.
\end{align*}
Taking $A>1$ large enough such that
$${\mathcal{E}(t)}\ge\norm{\u}{H^{r+\alpha+3}}^2+\norm{\b}{H^{r+\alpha+3}}^2.$$
Then, choosing $\delta>0$ small enough, we can get from \eqref{ping30} that
\begin{align}\label{ping33}
\frac{d}{dt}{\mathcal{E}(t)}+\frac12{\mathcal{D}(t)}\le 0.
\end{align}
For any ${N}\ge (2\beta+3)r+\alpha+2\beta+5$, we invoke the interpolation inequality
\begin{align*}
	\norm{ \u}{H^{r+\alpha+3}}\le&\norm{ \u}{H^{\alpha+2}}^{\frac{2(\beta+1)}{2\beta+3}}\norm{ \u}{H^{{N}}}^{\frac{1}{2\beta+3}}\le C\delta^{\frac{1}{2\beta+3}}\norm{{\n}\cdot\nabla \u}{H^{r+\alpha+2}}^{\frac{2(\beta+1)}{2\beta+3}}
\end{align*}
which further implies that
\begin{align*}
{\mathcal{E}(t)}\le& C(\norm{\u}{H^{r+\alpha+3}}^2+\norm{\b}{H^{r+\alpha+3}}^2)\nn\\
\le& C\norm{ \u}{H^{\alpha+2}}^{\frac{4(\beta+1)}{2\beta+3}}\norm{ \u}{H^{{N}}}^{\frac{2}{2\beta+3}}+C\norm{ \b}{H^{\alpha+2}}^{\frac{4(\beta+1)}{2\beta+3}}\norm{ \b}{H^{{N}}}^{\frac{2}{2\beta+3}}\nn\\
\le& C\delta^{\frac{2}{2\beta+3}}\norm{{\n}\cdot\nabla \u}{H^{r+\alpha+2}}^{\frac{4(\beta+1)}{2\beta+3}}+C\delta^{\frac{2}{2\beta+3}}\norm{\nabla \b}{H^{r+\alpha+2}}^{\frac{4(\beta+1)}{2\beta+3}}\nn\\
\le&C(D(t))^{{\frac{2(\beta+1)}{2\beta+3}}}.
\end{align*}
So, we get a  Lyapunov-type differential inequality
\begin{align*}
\frac{d}{dt}{\mathcal{E}(t)}+c({\mathcal{E}(t)})^{{\frac{2\beta+3}{2(\beta+1)}}}\le 0.
\end{align*}

Solving this inequality yields
\begin{align}\label{ping37}
{\mathcal{E}(t)}\le C(1+t)^{-2(\beta+1)}.
\end{align}
Taking $m={N}$ in \eqref{3.16} and using the embedding relation give
\begin{align}\label{ping38}
&\frac{d}{dt}(\norm{\u}{H^{N}}^2+\norm{\b}{H^{N}}^2)+\norm{\nabla\b}{H^{N}}^2\le C(\norm{ \u}{H^3}+\norm{\b}{H^3})(\norm{\u}{H^{N}}^2+\norm{\b}{H^{N}}^2).
\end{align}
From \eqref{ping37}, we have
\begin{align*}
\int_0^t(\norm{ \u(\tau)}{H^3}+\norm{ \b(\tau)}{H^3})\,d\tau\le C,
\end{align*}
thus, exploiting
the Gronwall inequality to \eqref{ping38} implies
\begin{align*}
\norm{\u}{H^{N}}^2+\norm{\b}{H^{N}}^2
\le&C(\norm{\u_0}{H^{N}}^2+\norm{\b_0}{H^{N}}^2)
\le C\varepsilon^2.
\end{align*}
Taking $\varepsilon$ small enough so that $C\varepsilon\le \delta/2$, we deduce from a continuity argument that the local solution
can be extended as a global one in time.

Moreover, from \eqref{ping37}, we also have the following decay rate
\begin{align*}
\norm{\u(t)}{H^{r+\alpha+3}}+\norm{\b(t)}{H^{r+\alpha+3}}\le C(1+t)^{-(\beta+1)}.
\end{align*}
Thus,
for any $\gamma> r+\alpha+3$, choosing ${N}>\gamma$ and  using the following interpolation inequality\begin{align*}
	\norm{f(t)}{H^{\gamma}}\le\norm{f(t)}{H^{r+\alpha+3}}^{\frac{{N}-\gamma}{{N}-r-\alpha-3}}
	\norm{f(t)}{H^{N}}^{\frac{\gamma-r-\alpha-3}{{N}-r-\alpha-3}}.
\end{align*}
we can get the decay rate for the higher order energy
\begin{align*}
\norm{\u(t)}{H^{\gamma}}+\norm{\b(t)}{H^{\gamma}}\le C(1+t)^{-\frac{({N}-\gamma)(\beta+1)}{{N}-r-\alpha-3}}.
\end{align*}
This completes the proof of Theorem \ref{dingli}.$\hspace{8.3cm}\square$

\section*{ Acknowledgments}
This work is  supported by the Guangdong Provincial Natural Science Foundation under grant 2022A1515011977 and the Science and Technology Program of Shenzhen under grant 20200806-104726001.

\bigskip
\noindent{Data availability statement:}
 Data sharing not applicable to this article as no datasets were generated or analysed during the
current study.

\end{document}